\documentclass[12pt]{article}
\usepackage{amsmath}
\usepackage{algorithmic}
\usepackage[Algorithm,ruled]{algorithm}
\usepackage{titlesec}
\usepackage{sectsty}
\usepackage[left=2cm,right=2cm,top=2cm,bottom=2cm]{geometry}
\usepackage{amssymb,amscd}
\usepackage{amsthm}
\usepackage{graphicx}
\usepackage[normalem]{ulem}
\usepackage[bookmarksnumbered,  plainpages]{hyperref}
\usepackage{multirow}
\usepackage{inputenc}

\newtheoremstyle{case}{}{}{}{}{}{:}{ }{}
\theoremstyle{case}

\newcommand{\be}{\begin{equation}}
\newcommand{\ee}{\end{equation}}
\newcommand{\ben}{\begin{eqnarray*}}
\newcommand{\een}{\end{eqnarray*}}
\newtheorem{examp}{\sc Example}
\newtheorem{remk}{\sc Remark}
\newtheorem{corol}{\sc Corollary}
\newtheorem{lemma}{\sc Lemma}
\newtheorem{theorem}{\sc Theorem}
\newtheorem{defn}{\sc Definition}
\newtheorem{prop}{\sc Proposition}
\newcommand{\bet}{\begin{theorem}}
    \newcommand{\eet}{\end{theorem}}
\newcommand{\bel}{\begin{lemma}}
    \newcommand{\eel}{\end{lemma}}
\newcommand{\bed}{\begin{defn}}
    \newcommand{\eed}{\end{defn}}
\newcommand{\brem}{\begin{remk}}
    \newcommand{\erem}{\end{remk}}
\newcommand{\bex}{\begin{examp}}
    \newcommand{\eex}{\end{examp}}
\newcommand{\bcl}{\begin{corol}}
    \newcommand{\ecl}{\end{corol}}
\newcommand{\bep}{\begin{prop}}
    \newcommand{\eep}{\end{prop}}

\newcommand{\NI}{\noindent}
\newcommand{\bea}{\begin{eqnarray}}
\newcommand{\eea}{\end{eqnarray}}

\newcommand{\vsp}{\vskip 1em}

\begin{document}
\title{\large {\bf{\sc Solution of oligopoly market equilibrium problem using modified newton method }}}
\author{ A. Dutta$^{a, 1}$ and  A. K. Das$^{b, 2}$\\
\emph{\small $^{a}$Department of Mathematics, Jadavpur University, Kolkata, 700 032, India}\\	
\emph{\small $^{b}$SQC \& OR Unit, Indian Statistical Institute, Kolkata, 700 108, India}\\
\emph{\small $^{1}$Email: aritradutta001@gmail.com}\\
\emph{\small $^{2}$Email: akdas@isical.ac.in} \\
 }
\date{}
\maketitle
\begin{abstract}
	The paper aims to find the solution of oligopoly market equilibrium problem through system of nonlinear equations. We propose modified newton method to obtain the solution of system of nonlinear equations. We show that our proposed method has higher order of convergence. 
	\vsp
\NI{\bf Keywords:} System of nonlinear equations, modified newton method, oligopolistic market equilibrium problem. \\
\end{abstract}
\section{Introduction}
 An oligopolistic 
market structure in which $n$ firms supply a homogeneous product in a noncooperative fashion. 

Following Murphy et al. \cite{murphy} the oligopolistic market equilibrium problem is as follows:

Let there be $n$ firms, which supply a homogeneous product in a noncooperative fashion. Let $P(\tilde{Q}), \tilde{Q} \geq 0$ denote the inverse demand, where $\tilde{Q}= \sum_{i=1}^{n} Q_i,$ $Q_i \geq 0$ denote the $i$-th firm's supply. Let $c_i(Q_i)$ be the total cost of supplying $Q_i$ units. Now the Nash equilibrium solution is a set of nonnegative output levels ${Q_1}^*,{Q_2}^*,\cdots,{Q_n}^*,$ such that ${Q_i}^*$ is an optimal solution to the following problem $\forall i\in \{1,2\cdots, n\}:$
\begin{eqnarray}\label{obj}
\text{maximize} \ {Q_iP(Q_i+{\tilde{Q}_i}^*)-c_i(Q_i)}
\end{eqnarray}
where ${\tilde{Q}_i}^*= \sum_{j \neq i}{Q_j}^*.$
Murphy et al. show that if $c_i(Q_i)$ is convex and continuously differentiable $\forall i \in\{1,2, \cdots,n\} $ and the inverse demand function $P(\tilde{Q})$ is strictly decreasing and continuously differentiable and the industry revenue curve $\tilde{Q}P(\tilde{Q})$ is concave, then \\
$({Q_1}^*,{Q_2}^*,\cdots,{Q_n}^*)$ is a Nash equilibrium solution if and only if
\begin{eqnarray}
[P(\tilde{Q}^*)+{Q_i}^*P'(\tilde{Q}^*)-{c_i}'({Q_i}^*)]{Q_i}^*=0 \\
{c_i}'({Q_i}^*)-P(\tilde{Q}^*)-{Q_i}^*P'(\tilde{Q}^*) \geq 0 \label{22}\\
{Q_i}^* \geq 0 \ \    \forall i\in \{1,2,\cdots,n\}
\end{eqnarray}

where $\tilde{Q}^*=\sum_{i=1}^{n}{Q_i}^*,$ which is a nonlinear complementarity problem with $f_i(z)={c_i}'({Q_i}^*)-P(\tilde{Q}^*)-{Q_i}^*P'(\tilde{Q}^*),  $ and $z_i= {Q_i}^*.$ \\
Note that here the functions $c_i(Q_i)$ and $-\tilde{Q}P(\tilde{Q})$ are convex. So the 1st order derivative of these two functions are increasing function. Hence the function $f_i(z)={c_i}'({Q_i}^*)-P(\tilde{Q}^*)-{Q_i}^*P'(\tilde{Q}^*)  $ is an increasing function.\\

The nonlinear complementarity problem is identified as an important mathematical programming problem and  has  been  used  as  a  general  framework  for  quadratic  programming, linear complementarity problems and some equilibrium problems. A number of applications of nonlinear complementarity problems are reported in operations research \cite{pang1995complementarity}, multiple objective programming problem \cite{kostreva1993linear}, mathematical economics and engineering.
The concept of complementarity is synonymous with the notion of system equilibrium. 
The nonlinear complementarity problem is well studied in the literature on mathematical
programming and arises in a number of applications in operations research, control theory,
mathematical economics, geometry and engineering. 

The idea of nonlinear complementarity problem is based on the concept of linear complementarity problem. For recent study on this problem and applications see
\cite{das2017finiteness}, \cite{articlee14}, \cite{bookk1}, \cite{articlee7}, \cite{b1}, \cite{b9}, \cite{b10} and references therein. For details of several matrix classes in complementarity theory, see \cite{articlee1}, \cite{articlee2}, \cite{articlee9}, \cite{articlee17}, \cite{article1}, \cite{mohan2001more}, \cite{b2}, \cite{b3} \cite{article07}, \cite{dutta2022column} and references cited therein. The problem of computing the value vector and optimal stationary strategies for structured stochastic games for discounted and undiscounded zero-sum games and quadratic multi-objective programming problem are formulated as linear complementary problems. For details see \cite{articlee18}, \cite{b4}, \cite{b5}, \cite{neogy2005linear} and \cite{neogy2008mixture}. The complementarity problems are considered with respect to principal pivot transforms and pivotal method to its solution point of view. For details see \cite{articlee8}, \cite{articlee10}, \cite{b6}, \cite{b7}, \cite{b8} and \cite{das1}.

A wide class of problems, which arise in various fields of sciences, can be studied via the nonlinear system of equations, $ f(x) = 0 $ using various techniques. Finding  solutions of systems of nonlinear equations has an important role to deal with problems in various fields such as chemical production processes,
engineering design, economic equilibrium, transportation and applied physics. A number of methods are proposed to solve systems of equations. Newton and quasi-newton  methods are  well known iterative methods to solve nonlinear systems of equations. 
 In recent years, researchers are interested  to solve system of nonlinear equations both analytically and numerically. Several iterative methods have been developed using different techniques such as Taylor’s series expansion,
quadrature formulas, homotopy method, interpolation, decomposition and its various modification.
For details, see \cite{sayevand2016systems}, \cite{noor2009some}, \cite{jafari}, \cite{ojika1983deflation}, \cite{fan2003modified} and \cite{yamamura1998interval}.\\
Suppose $f:R\to R,$ a nonlinear function. Then the equation $f(x)=0$ can be solved by newton method with the iterative process $x^{k+1}=x^{k}-f'(x^{k})^{-1}f(x^{k}).$ Consider the nonlinear function $g:R^n \to R^n$ . Then the system of nonlinear equations $g(x)=0$ can be solved by newton method with the iterative process $x^{k+1}=x^{k}-J_g (x^{k})^{-1}g(x^{k}),$ where $J_g$ is the jacobian of the function $g(x)=0.$ Many researchers extended newton method in various way to get the solution with higher order of convergence . For details see \cite{homeier2004modified}, \cite{kou2006modification}, \cite{ehrlich1967modified}.

\section{Preliminaries}
We begin by introducing some basic notations used in this paper. $R^n$  denotes the $n$ dimensional real space, $R^n_+$ and $R^{n}_{++}$ denote the nonnegative and positive orthant of $R^n.$ We consider vectors and matrices with real entries. Any vector $x\in R^{n}$ is a column vector and  $x^{T}$ denotes the row transpose of $x.$ $e$ denotes the vector of all $1.$\\
The nonlinear complementarity problem is to find the vector $z\in R^n,$ such that $f(z),z \geq 0, \ z^Tf(z)=0,$ where $f(z)$ is a nonlinear equation.
\section{Main Results}
Now we show that the nonlinear complementarity problem can be solved by system of nonlinear equations.
   \begin{theorem} 	Let $\phi:R \to R$ be any increasing function such that $\phi(0)=0$. Then $z$ solves the complementarity problem\ref{cp} if and only if 
    	\begin{equation}\label{ma}
    	\phi((f_i(z)-z_i)^2) - \phi(f_i(z) |f_i(z)|) - \phi(z_i|z_i|)=0
    	\end{equation}
       \end{theorem}
   \begin{proof}
   	Necessery. For each $i=1,2,\ldots ,n$, either $z_i=0 , f_i(z)\geq 0$  or $f_i(z)=0 , z_i \geq 0$.\\ If $z_i=0,f_i(z)\geq 0 $ then
   	$\phi((f_i(z)-z_i)^2) - \phi(f_i(z) |f_i(z)|) - \phi(z_i|z_i|) $\\
   	$ = \phi((f_i(z))^2) - \phi(f_i(z) |f_i(z)|)=\phi((f_i(z))^2) - \phi((f_i(z))^2)=0$.\\
   	If $f_i(z)=0 , z_i \geq 0$ then  $\phi((f_i(z)-z_i)^2) - \phi(f_i(z) |f_i(z)|) - \phi(z_i|z_i|) $\\ 
   	 	$ = \phi((z_i)^2) - \phi(z_i |z_i|)=\phi((z_i)^2) - \phi((z_i)^2)=0$.\\
   	 	So the solution of $(1.1)$ satisfies $(1.2)$.
   	 	\\
   	 	Sufficient. (a) To show that $f(z) \geq 0$ assume the contrary, that is, $f_i(z) < 0$ for some $i=1,2,\dots,n$. Then
   	 	 \begin{center}
   	 	 $0\leq \phi((f_i(z)-z_i)^2)$
   	 	$= \phi(f_i(z)|f_i(z)|) + \phi(z_i|z_i|)$ $= \phi(-f_i(z)^2) + \phi(z_i|z_i|)$ $ < \phi(z_i|z_i|)$
   	 \end{center}
      This implies that $ \phi (z_i|z_i|) > 0$ $ \implies z_i|z_i| >0 $ $ \implies z_i >0$ and  $\phi((f_i(z)-z_i)^2)< \phi(z_i|z_i|)$ $\implies ((f_i(z)-z_i)^2) < z_i|z_i|$.\\
   	 	But for $z_i >0$,  $(z_i)^2 < (f_i(z)-z_i)^2$   [as $(f_i(z)-z_i)<0, z_i >0 $ then $ (f_i(z)-z_i)<0$, $|f_i(z)-z_i|>z_i$ so $(f_i(z)-z_i)^2 >z_i^2$ ].\\
   	 	 So, it contradicts that $f_i(z) < 0$. So $f_i(z) \geq 0$.\\
   	 	 (b)  To show that $z \geq 0$ interchange the roles of $ z_i$ and $f_i(z)$.\\
   	 	 (c)From (a) and (b) we have that  $z \geq 0$ and  $f(z) \geq 0$. To show $z^Tf(z)=0$ assume the contrary $z_i >0$ and $f_i(z)>0$ for some $ i =1,2,\ldots,n.$\\
   	 	 If $f_i(z)\geq z_i$, then $\phi( (f_i(z)-z_i)^2) < \phi( (f_i(z))^2) + \phi( (z_i)^2)= \phi(f_i(z) |f_i(z)|) + \phi(z_i|z_i|)$. This contradicts that 	$\phi((f_i(z)-z_i)^2) = \phi(f_i(z) |f_i(z)|) + \phi(z_i|z_i|) $\\
   	 	 Similarly, for  $z_i\geq f_i(z),$ just interchange the roles of $z_i$ and $f_i(z)$ and will get a contradiction.
   \end{proof}
   Hence it is shown that the complementarity problem of finding a $z\in R^n$ satisfying $z^Tf(z)=0,\ \ \ f(z) \geq 0, \ \ \ z \geq 0$ where  $f: R^n \rightarrow R^n$, a nonlinear equation, is equivalent to the problem of solving system of n nonlinear equations in n variables.
  	\begin{equation}{\label{ne}}
  \psi_i(z)=\phi((f_i(z)-z_i)^2) - \phi(f_i(z) |f_i(z)|) - \phi(z_i|z_i|)=0 \ \ \ \forall   i \in \{1,2,\cdots n\},
  \end{equation}
  where $\phi$ is an increasing function defined by $\phi:R \to R$ such that $\phi(x)=x^3.$
   Now $\frac{\partial \psi_i}{\partial z_j}$ $=\phi^{'}((f_i(z)-z_i)^2)2(f_i(z)-z_i)(\frac{\partial f_i}{\partial z_j}-\delta_{ij})$  $-\phi^{'}(f_i(z)|f_i(z)|)2f_i(z)sgn(f_i(z))\frac{\partial f_i}{\partial z_j}$  $-\phi^{'}(z_i|z_i|)2z_isgn(z_i) \delta_{ij}$
    where 
    \[   
    sgn(x) = 
    \begin{cases}
    \  1 &\quad\text{if } x  > 0\\
    \ -1 &\quad\text{if } x < 0\\
    \end{cases}
    \]
     \quad

\subsection{Modified Newton method to solve system of nonlinear equations}

Now we introduce the modified Newton's method to solve the  system of nonlinear equations. The algorithm of the modified newton's method is given below.\\
\textbf {Algorithm:}\label{newmodi}\\

{ \textbf{\textit{Modified Newton Method}}\label{allgo}

\NI\textbf{Step 0:} 
Give the initial approximation $z_0$ and a very small number $e, $ such that $0<e.$  Compute the Jacobian of $\psi(z)$ with respect to the variable $z,$ which is denoted by  $\cal{J}$ and $\psi'(z)=\cal{J}.$ Set $k=0.$
\\
 \textbf{Step 1:}  For $k$-th iteration compute the followings:\\
  \begin{equation}\label{eq1}
      \begin{aligned} 
 y^k=z^k-\frac{1}{2}[\psi'(z^k)+diag(t_i\psi_i(z^k)]^{-1}\psi(z^k);\\
    x^k=z^k-\frac{1}{2}[\{\psi'(z^k)\}^2+\{\psi'(y^k)\}^2+diag(\lambda_i\{\psi_i(z^k)\}^2]^{-1} \\
   [\psi'(z^k)+ \psi'(y^k)]\psi(z^k);\\
 w^k=x^k-[\{\psi'(x^k)\}^2+\{\psi'(y^k)\}^2+diag(\mu_i\{\psi_i(x^k)\}^2]^{-1} \\
 [\psi'(x^k)+
 \psi'(y^k)]\psi(x^k);\\
  z^{k+1}=w^k-[\psi'(w^k)+diag(\eta_i\{\psi_i(w^k)\}^2]^{-1}\psi(w^k),
      \end{aligned}
       \end{equation}\\
       where $\psi'(x)$ is the jacobian of $\psi(x)$ and\\
       sgn$(t_i \psi_i (z^k))=$sgn$(\frac{\partial \psi_i}{\partial z_i}(z^k)),$\\
      sgn$(\lambda_i)=$sgn$(\frac{\partial \psi_i}{\partial z_i}(z^k)),$\\
      sgn$(\mu_i)=$sgn$(\frac{\partial \psi_i}{\partial x_i}(x^k)),$\\
      sgn$(\eta_i)=$sgn$(\frac{\partial \psi_i}{\partial w_i}(w^k)).$\\
\NI \textbf{Step 2:}  Compute the norm of $\psi(z)=0$ which is denoted by $n_1.$ Here $n_1$ is defined by $n_1=\|\psi(z)\|=\sqrt{\sum_{i=1}^{n}\{\psi_i (z)\}^2}$.\\
 \textbf{Step 3:}
     If the norm $n_1$ is less than the tolerance $e,$ i.e. $n_1<e,$ then the new iteration i.e. $z^{k+1}$ is the required solution of $\psi(z)=0,$ i.e. $\psi(z^{k+1})=  \begin{bmatrix} \psi_1(z^{k+1})\\\psi_2(z^{k+1})\\ \vdots \\ \psi_n(z^{k+1}) \end{bmatrix}$ $=\begin{bmatrix} 0\\0\\ \vdots \\0 \end{bmatrix}.$ Otherwise set $k=k+1$ and go to step $1.$
     \bigskip
     \begin{theorem}
         Let $\psi : R^n \to R^n$ has a root  $z^* \in D \subseteq R^n$, where $D$ is an open convex set. Assume that  $\psi(z)$ is three times Fre{'}chet differentiable in some neighborhood $N$ of the root $(z^*)$. If for all $z \in N$, $diag(t_i\psi_i(z^k), diag(\lambda_i\{\psi_i(z^k)\}^2,\\ diag(\mu_i\{\psi_i(x^k)\}^2$ and $diag(\eta_i\{\psi_i(w^k)\}^2$  are nonsingular, then the method defined by \ref{eq1} is of seventh-order convergence.
     \end{theorem}
     \begin{proof}
         Let $e^k=z^k-z^*.$ Now using Taylor series expansion we have,\\
         $\psi(z^*)=\psi(z^k)+\psi'(z^k)(z^*-z^k)+\frac{1}{2}\psi''(z^k)(z^*-z^k)^2+o(\|e^k\|^3)$\\ $\implies \psi(z^k)=\psi'(z^k)e^k-\frac{1}{2}\psi''(z^k)(e^k)^2+o(\|e^k\|^3).$\\
         Let $d^k=y^k-z^k=-\frac{1}{2}[\psi'(z^k)+diag(t_i\psi_i(z^k))]^{-1}\psi(z^k).$ \\ Now using Taylor series expansion we have,\\$\psi(y^k)=\psi(z^k)+\psi'(z^k)d^k+\frac{1}{2}\psi''(z^k)(d^k)^2+o(\|d^k\|^3).$\\ Now $d^k=y^k-z^k=-\frac{1}{2}[\psi'(z^k)+diag(t_i\psi_i(z^k))]^{-1}\psi(z^k)\\=-\frac{1}{2}[\psi'(z^k)+diag(t_i\psi_i(z^k))]^{-1}\psi'(z^k)e^k+o(\|e^k\|^2)\\$
         $=-\frac{1}{2}[\psi'(z^k)+diag(t_i\psi_i(z^k))]^{-1}[\psi'(z^k)+diag(t_i\psi_i(z^k))]e^k+\\ \frac{1}{2}[\psi'(z^k)+diag(t_i\psi_i(z^k))]^{-1}[diag(t_i\psi_i(z^k))]e^k+o(\|e^k\|^2)$\\ $=-\frac{1}{2}e^k+o(\|e^k\|^2),$ as $diag(t_i\psi_i(z^k))]e^k=o(\|e^k\|^2).$\\
         Hence $\psi(y^k)=\psi(z^k)-\frac{1}{2}\psi'(z^k)e^k+o(\|e^k\|^2)=\psi'(z^k)e^k-\frac{1}{2}\psi'(z^k)e^k+o(\|e^k\|^2)=\frac{1}{2}\psi'(z^k)e^k+o(\|e^k\|^2)$ and $\psi'(y^k)=\frac{1}{2}\psi'(z^k)+o(\|e^k\|).$\\
         Let $a^k=x^k-z^*. $ Hence  $x^k=z^k-\frac{1}{2}[\{\psi'(z^k)\}^2+\{\psi'(y^k)\}^2+diag(\lambda_i\{\psi_i(z^k)\}^2)]^{-1}[\psi'(z^k)+ \psi'(y^k)]\psi(z^k)$ implies that\\ $a^k=e^k-\frac{1}{2}[\{\psi'(z^k)\}^2+\{\psi'(y^k)\}^2+diag(\lambda_i\{\psi_i(z^k)\}^2)]^{-1}[\psi'(z^k)+ \psi'(y^k)]\psi(z^k)$\\
         $\implies a^k= [\{\psi'(z^k)\}^2+\{\psi'(y^k)\}^2+diag(\lambda_i\{\psi_i(z^k)\}^2)]^{-1}\{[\{\psi'(z^k)\}^2+\{\psi'(y^k)\}^2+diag(\lambda_i\{\psi_i(z^k)\}^2)]e^k-\frac{1}{2}[\psi'(z^k)+ \psi'(y^k)]\psi(z^k)\}\\$ $=[\{\psi'(z^k)\}^2+\{\psi'(y^k)\}^2+diag(\lambda_i\{\psi_i(z^k)\}^2)]^{-1}\{[diag(\lambda_i\{\psi_i(z^k)\}^2)]e^k+\psi'(z^k)\}^2 e^k+\psi'(y^k)\}^2 e^k-\frac{1}{2}[\psi'(z^k)+ \psi'(y^k)]\psi(z^k)\}=o(\|e^k\|^3),$ as $diag(\lambda_i\{\psi_i(z^k)\}^2) e^k=o(\|e^k\|^3).$\\
         Again using Taylor series expansion we have,\\$\psi(x^k)=\psi(z^k)+\psi'(z^k)c^k+\frac{1}{2}\psi''(z^k)(c^k)^2+o(\|c^k\|^3),$where\\
         $c^k=x^k-z^k=-\frac{1}{2}[\{\psi'(z^k)\}^2+\{\psi'(y^k)\}^2+diag(\lambda_i\{\psi_i(z^k)\}^2)]^{-1}[\psi'(z^k)+ \psi'(y^k)]\psi(z^k)$ $=-\frac{1}{2}[\{\psi'(z^k)\}^2+\{\psi'(y^k)\}^2+diag(\lambda_i\{\psi_i(z^k)\}^2)]^{-1}[\psi'(z^k)+ \psi'(y^k)][\psi'(z^k)e^k]+o(\|e^k\|^2)$ $=-\frac{1}{2}[\{\psi'(z^k)\}^2+\{\psi'(y^k)\}^2+diag(\lambda_i\{\psi_i(z^k)\}^2)]^{-1}[\{\psi'(z^k)\}^2 e^k+ \psi'(y^k)\psi'(z^k)e^k]+o(\|e^k\|^2)$ $=-\frac{1}{2}[\{\psi'(z^k)\}^2+\{\psi'(y^k)\}^2+diag(\lambda_i\{\psi_i(z^k)\}^2)]^{-1}[\{\psi'(z^k)\}^2+\{\psi'(y^k)\}^2+$\\$diag(\lambda_i\{\psi_i(z^k)\}^2)] e^k+\frac{1}{2}[\{\psi'(z^k)\}^2+\{\psi'(y^k)\}^2+$\\$diag(\lambda_i\{\psi_i(z^k)\}^2)]^{-1} [\{\psi'(y^k)\}^2+diag(\lambda_i\{\psi_i(z^k)\}^2)-\psi'(y^k)\psi'(z^k)]e^k+o(\|e^k\|^2)=-\frac{1}{2}e^k+o(\|e^k\|^2).$\\
          Hence $\psi(x^k)=\psi(z^k)-\frac{1}{2}\psi'(z^k)e^k+o(\|e^k\|^2)=\psi'(z^k)e^k-\frac{1}{2}\psi'(z^k)e^k+o(\|e^k\|^2)=\frac{1}{2}\psi'(z^k)e^k+o(\|e^k\|^2)$ and $\psi'(x^k)=\frac{1}{2}\psi'(z^k)+o(\|e^k\|).$\\
        
         Let $b^k=w^k-z^*=x^k-z^*-\frac{1}{2}[\{\psi'(x^k)\}^2+\{\psi'(y^k)\}^2+diag(\mu_i\{\psi_i(x^k)\}^2)]^{-1} [\psi'(x^k)+ \psi'(y^k)]\psi(x^k)=a^k-\frac{1}{2}[\{\psi'(x^k)\}^2+\{\psi'(y^k)\}^2+diag(\mu_i\{\psi_i(x^k)\}^2)]^{-1} [\psi'(x^k)+ \psi'(y^k)]\psi(x^k)$\\ $\implies b^k=[\{\psi'(x^k)\}^2+\{\psi'(y^k)\}^2+diag(\mu_i\{\psi_i(x^k)\}^2)]^{-1}\{[\{\psi'(x^k)\}^2+\{\psi'(y^k)\}^2+diag(\mu_i\{\psi_i(x^k)\}^2)]a^k-[\psi'(x^k)+ \psi'(y^k)]\psi(x^k)\}=o(\|e^k\|^5),$ as $diag(\mu_i\{\psi_i(x^k)\}^2) a^k=o(\|e^k\|^5).$ Now using Taylor series expansion we have,\\$\psi(w^k)=\psi(z^k)+\psi'(z^k)m^k+\frac{1}{2}\psi''(z^k)(m^k)^2+o(\|m^k\|^3),$ where
         $m^k=w^k-z^k.$ \\ Let $n^k=w^k-x^k=-[\{\psi'(x^k)\}^2+\{\psi'(y^k)\}^2+diag(\mu_i\{\psi_i(x^k)\}^2)]^{-1}[\psi'(x^k)+ \psi'(y^k)]\psi(x^k)=-[\{\psi'(x^k)\}^2+\{\psi'(y^k)\}^2+diag(\mu_i\{\psi_i(x^k)\}^2)]^{-1}[\frac{1}{2}\psi'(z^k)+ \frac{1}{2}\psi'(z^k)][\frac{1}{2}\psi'(z^k)e^k]+o(\|e^k\|^2)$ $=-\frac{1}{2}[\frac{1}{4}\{\psi'(z^k)\}^2+\frac{1}{4}\{\psi'(z^k)\}^2+diag(\mu_i\{\psi_i(x^k)\}^2)]^{-1}[\{\psi'(z^k)\}^2 e^k]+o(\|e^k\|^2)$ $=-\frac{1}{2}[\{\psi'(z^k)\}^2+diag(\mu_i\{\psi_i(x^k)\}^2)]^{-1}[\{\psi'(z^k)\}^2 e^k]+o(\|e^k\|^2)=-\frac{1}{2}[\{\psi'(z^k)\}^2+diag(\mu_i\{\psi_i(x^k)\}^2)]^{-1}[\{\psi'(z^k)\}^2+diag(\mu_i\{\psi_i(x^k)\}^2)]e^k+\frac{1}{2}[\{\psi'(z^k)\}^2+\\diag(\mu_i\{\psi_i(x^k)\}^2)]^{-1}[diag(\mu_i\{\psi_i(x^k)\}^2) e^k]+o(\|e^k\|^2)=-\frac{1}{2}e^k+o(\|e^k\|^2).$ Hence $m^k=w^k-z^k=w^k-x^k+x^k-z^k=n^k+c^k=-e^k+o(\|e^k\|^2).$ Therefore $\psi(w^k)=\psi(z^k)-\psi'(z^k)e^k+o(\|e^k\|^2).$ \\
         Now $z^{k+1}-z^*=w^k-z^*-[\psi'(w^k)+diag(\eta_i\{\psi_i(w^k)\}^2]^{-1}\psi(w^k)$\\ $\implies  e^{k+1}=b^k-[\psi'(w^k)+diag(\eta_i\{\psi_i(w^k)\}^2]^{-1}\psi(w^k)=[\psi'(w^k)+diag(\eta_i\{\psi_i(w^k)\}^2]^{-1}\{[\psi'(w^k)+diag(\eta_i\{\psi_i(w^k)\}^2]b^k-\psi(w^k)\}=[\psi'(w^k)+diag(\eta_i\{\psi_i(w^k)\}^2]^{-1}\{\psi'(w^k)b^k+diag(\eta_i\{\psi_i(w^k)\}^2 b^k-\psi(w^k)\}=o(\|e^k\|^7),$ as $diag(\eta_i\{\psi_i(w^k)\}^2 b^k=o(\|e^k\|^7)$. Therefore the introduced modified Newton method has seventh order of convergency.
         
     \end{proof}

\section{Numerical Example}\label{oliex}
Now consider an oligopoly with five firms, each with a total cost function of the form:
\begin{equation}\label{olix}
c_i(Q_i)=n_iQ_i+\frac{\beta_i}{\beta_i+1}{L_i}^{\frac{1}{\beta_i}}{Q_i}^{\frac{\beta_i+1}{\beta_i}}
\end{equation}
The demand curve is given by:
\begin{equation}
\tilde{Q}=5000P^{-1.1}, \ \   P(\tilde{Q})=5000^{1/1.1}\tilde{Q}^{-1/1.1}.
\end{equation}

The parameters of the equation \ref{olix} for the five firms are given below: 
\begin{table}[ht]
	\caption{Value of parameters for five firms} 
	\centering
	\begin{tabular}{c c c c} 
		\hline\hline
		firm $i$ & $n_i$ & $L_i$ & $\beta_i$ \\ [0.5ex] 
		\hline
		1 & 10 & 5 & 1.2 \\ 
		2 & 8 & 5 & 1.1 \\
		3 & 6 & 5 & 1 \\
		4 & 4 & 5 & 0.8 \\
		5 & 2 & 5 & 0.6 \\ [1ex] 
		\hline
	\end{tabular}
\end{table}\\

\newpage
 This oligopoly market equilibrium problem is a nonlinear complementarity problem $z\geq 0, f(z)\geq 0, z^Tf(z)=0$. The nonlinear complementarity problem is solved by solving system of nonlinear equations 
 \begin{equation}\label{ne}
     \phi((f_i(z)-z_i)^2) - \phi(f_i(z) |f_i(z)|) - \phi(z_i|z_i|)=0
 \end{equation}	 where $\phi(z)=z^3.$\\
 Now we solve the system of nonlinear equations \ref{ne} by the above proposed algorithm.

 	 Now to solve this oligopoly problem by modified newton's method \ref{newmodi}, first take the initial point $z_0=$ $\left[\begin{array}{c} 
 	40\\
 	50\\
 	60\\
 	55\\
 	45\\
 	\end{array}\right].$ Set $e=10^{-7}, n_1=10^{-9}.$ After $17$ iterations we obtain  $z=$ $\left[\begin{array}{c} 
 	15.4293\\
 	12.4986\\
 	9.6635\\
 	7.1651\\
 	5.1326\\
 	\end{array}\right],$
 	
\section{Conclusion}
In this study we consider oligopolistic market equilibrium problem and give solution approach through system of nonlinear equation. An oligopolistic market equilibrium problem can be reduced to a nonlinear complementarity problem. We show the equivalency between the nonlinear complementarity problem and the system of nonlinear equations.  We introduce modified Newton method with higher order of convergence to solve the system of nonlinear equations to obtain the solution of nonlinear complementarity problem. Although the Jacobian of the system of nonlinear equations be singular, this method can be processed with suitable initial point. We show that the rate of convergency of this method is $7$.  Finally a real life oligopolistic market equilibrium problem is considered to demonstrate the effectiveness of our results and algorithms.
\section{Acknowledgment}
The author A. Dutta is thankful to the Department of Science and Technology, Govt. of India, INSPIRE Fellowship Scheme for financial support.

\bibliographystyle{plain}
\bibliography{name}
\end{document}